\documentclass[12pt]{article}

\usepackage[latin1]{inputenc}
\usepackage{amssymb,amsmath,amsfonts, amsthm}
\usepackage{enumerate}
\usepackage{times}
\usepackage{xcolor}
\usepackage{bbding}
\usepackage[margin=2.0cm]{geometry}
\usepackage{cite}
\usepackage{xcolor}
\usepackage{tikz}
\usetikzlibrary{patterns}

\usepackage{authblk}

\providecommand{\keywords}[1]{\textbf{\textit{Keywords}} #1}
\providecommand{\keywordsbis}[1]{\textbf{\textit{2010 Mathematics Subject Classification}} #1}

\numberwithin{equation}{section}

\theoremstyle{definition}
\newtheorem{theorem}{Theorem}[section]
\newtheorem{lemma}[theorem]{Lemma}

\newtheorem{proposition}[theorem]{Proposition}
\newtheorem{corollary}[theorem]{Corollary}

\newtheorem{remark}[theorem]{Remark}

\theoremstyle{remark}

\newcommand{\K}{{\mathcal K}}

\newcommand{\F}{{\mathcal{F}}}
\newcommand{\FF}{{\mathbf V}}

\newcommand{\R}{{\mathbb R}}
\newcommand{\N}{{\mathbb N}}

\newcommand{\sfe}{{\mathbb S}^{n-1}}

\newcommand{\Sym}{{\rm Sym}}
\renewcommand{\d}{\,\mathrm{d}}

\newcommand{\scdp}{C^{2,+}(\sfe)}
\newcommand{\scdpo}{C^{2,+}_0(\sfe)}

\title{On $p$-Brunn-Minkowski inequalities\\
for intrinsic volumes, with $0\leq p<1$}
\author{Chiara Bianchini\footnote{\textit{E-mail address:} \texttt{chiara.bianchini@unifi.it}}, Andrea Colesanti\footnote{\textit{E-mail address:} \texttt{andrea.colesanti@unifi.it}}, Daniele Pagnini\footnote{\textit{E-mail address:} \texttt{daniele.pagnini@unifi.it}},  Alberto Roncoroni\footnote{\textit{E-mail address:} \texttt{alberto.roncoroni@unifi.it}}}
\affil{\emph{Dipartimento di Matematica e Informatica ``Ulisse Dini''}\\\emph{Universit\`a degli Studi di Firenze, Viale Morgagni 67/A, 50134 Firenze, Italy}}

\date{}

\begin{document}

\maketitle

\begin{abstract}
We prove the validity of the $p$-Brunn-Minkowski inequality for the intrinsic volume $V_k$, $k=2,\dots, n-1$, of convex bodies in $\R^n$, in a neighborhood of the unit ball, for $0\le p<1$. We also prove that this inequality does not hold true on the entire class of convex bodies of $\R^n$, when $p$ is sufficiently close to $0$.

\end{abstract}
	
\bigskip	
	
\keywords{Convex bodies, Brunn-Minkowski inequality, intrinsic volumes}

\medskip

\keywordsbis{52A40, 52A20 (primary); 26D10 (secondary)}


\section{Introduction}
The Brunn-Minkowski inequality is one of the cornerstones of convex geometry, the branch of mathematics which studies the theory of convex bodies; in one of its formulations,  it states that the volume functional $V_n$ is $(\frac 1n)$-concave, that is
\begin{equation}\label{BM}
V_n((1-t)K_0+tK_1)^{1/n}\ge(1-t)V_n(K_0)^{1/n}+tV_n(K_1)^{1/n},\quad\forall\ K_0,K_1\in\K^n,\,\forall\ t\in[0,1],
\end{equation}
and equality holds if and only if $K_0$ and $K_1$ are homothetic, or contained in parallel hyperplanes. 
Here $V_n$ is the \emph{volume}, that is, the Lebesgue measure in $\R^n$, $K_0,K_1$ belong to the set of compact and convex subsets (\emph{convex bodies}) of $\R^n$, denoted by  $\K^n$, and the ``sum'' of sets indicates the \emph{Minkowski linear combination}, that is, the vectorial sum.
We refer the reader to the survey paper \cite{Gar}, and to the monograph \cite{Schneider} for a thorough presentation of the Brunn-Minkowski inequality, its numerous connections to other areas of mathematics, and applications.

\medskip

Inequality \eqref{BM} has a great number of variations and generalizations which consider different kinds of sums and different kinds of shape functionals; here we are interested 
in extending it to the so called \emph{$p$-addition} of convex bodies and to the \emph{intrinsic volumes}, rather than the classical volume.

\medskip

The $p$-sum of convex bodies was introduced by Firey in \cite{Firey} for $p\ge 1$, and offers an extension to the Minkowski sum (which represents the case $p=1$). Its definition is based on the \emph{support function} of a convex body $K$, which is denoted by $h_K\colon\sfe\to\R$ (see Section \ref{Preliminaries} for definitions and preliminary results) and finds its motivation starting from the behaviour of the support function with respect to the Minkowski sum of sets.
More precisely: for every $K,L\in\K^n$ and $\alpha, \beta\ge0$ the following equality holds:
$$
h_{\alpha K+\beta L}=\alpha h_K+\beta h_L. 
$$  
This relation motivates the definition of $p$-addition, for $p\ge1$: for $K,L\in\K^n$, both containing the origin (that is, $K,L\in\K^n_0$), and for $\alpha,\beta\ge 0$, the $p$-combination $\alpha\cdot K+_p\beta\cdot L$, with $p\ge 1$, is defined as the convex body whose support function is given by 
\begin{equation*}
h_{\alpha\cdot K+_p\beta\cdot L}=(\alpha h_K^p+\beta h_L^p)^{1/p}. 
\end{equation*}

This definition is well posed since $(\alpha h_K^p+\beta h_L^p)^{1/p}$ is a (non-negative) support function, by the condition $p\ge1$. 
The $p$-addition is at the core of the branch of convex geometry currently known as $L_p$-Brunn-Minkowski theory (see \cite[Chapter 9]{Schneider}), which received a major impulse by the works of Lutwak (see for instance \cite{Lutwak1, Lutwak2}).

A recent breakthrough in this context is due to the works \cite{BLYZ, BLYZ-1} by B\"or\"oczky, Lutwak, Yang and Zhang, where the authors begin the analysis of the range $p<1$, focusing on the case $p=0$. In particular, in \cite{BLYZ} they establish the following form of the Brunn-Minkowski inequality for the case $p=0$, called the log-Brunn-Minkowski inequality, which we state in Theorem \ref{teoBM0}. Given $K_0$ and $K_1$ in $\K^n_0$, and $t\in[0,1]$, consider the function $h_t\colon\sfe\to\R^n$ defined by
$$
h_t:=h_{K_0}^{1-t}h_{K_1}^t.
$$
Then define the convex body $(1-t)\cdot K_0+_0 t\cdot K_1$ as the \emph{Aleksandrov body}, or Wulff shape, of $h_t$; that is:
\begin{equation}\label{def0somma}
(1-t)\cdot K_0+_0 t\cdot K_1:=\{x\in\R^n\colon (x,y)\le h_t(y)\;\forall\ y\in\sfe\},
\end{equation}
where $ (\cdot,\cdot) $ denotes the standard scalar product of $ \R^n $. 

\begin{theorem}[B\"or\"oczky, Lutwak, Yang, Zhang] \label{teoBM0} For every $K_0, K_1\in\K^2_0$, origin symmetric, and for every $t\in[0,1]$:
$$
V_2((1-t)\cdot K_0+_0 t\cdot K_1)\ge V_2(K_0)^{1-t} V_2(K_1)^t.
$$
Equality holds if and only if $K_0$ and $K_1$ are dilates of each other, or they are parallelotopes. 
\end{theorem}

In \cite{BLYZ} the authors conjectured the same result to be valid in arbitrary dimension; this is the well known log-Brunn-Minkowski inequality conjecture.

\noindent{\bf Conjecture (Log-Brunn-Minkowski inequality - B\"or\"oczky, Lutwak, Yang, Zhang).} For every $K_0, K_1\in\K^n_0$, origin symmetric, and for every $t\in[0,1]$:
\begin{equation}\label{BM0}
V_n((1-t)\cdot K_0+_0 t\cdot K_1)\ge V_n(K_0)^{1-t} V_n(K_1)^t.
\end{equation}

The idea used in (\ref{def0somma}) to define the $0$-sum has been extended to define the $L_p$ convex linear combination of $K_0, K_1\in\K^n_0$, for $p\in(0,1)$. Given $t\in[0,1]$, we set
$$
(1-t)\cdot K_0+_p\ t\cdot K_1:=\{x\in\R^n\colon (x,y)\le \big((1-t)h_{K_0}^p(y)+t h_{K_1}^p(y)\big)^{1/p},\;\forall\ y\in\sfe\}.
$$ 
Note that, by standard properties of $p$-means, for every $p\ge 0$ it holds:
\begin{equation}\label{inclusion}
(1-t)\cdot K_0+_0 t\cdot K_1\subseteq
(1-t)\cdot K_0+_p\ t\cdot K_1,\quad\forall\ K_0,K_1\in\K^n_0,\;\forall\ t\in[0,1],
\end{equation}
(see \eqref{inclusion_bis} below). 
Hence, by the previous inclusion 
and by a standard argument based on the homogeneity of the volume, inequality \eqref{BM0} would imply:
\begin{equation}\label{BMp}
V_n((1-t)\cdot K_0+_p t\cdot K_1)^{p/n}\ge (1-t)V_n(K_0)^{p/n}+t V_n(K_1)^{p/n},
\end{equation}
for every $p\ge0$, that is a Brunn-Minkowski type inequality for the $p$-sum, for $p\ge 0$.

The conjectures about the validity of \eqref{BM0} and \eqref{BMp} originated an intense activity in recent years, and much progress has been made in this area (see \cite{Ma,Xi-Leng,Saroglou,Rotem,CLM,CL,Kolesnikov-Milman,Chen-Huang-Li-Liu,Putterman,Boroczky-Kalantzopoulos,Hosle-Kolesnikov-Livshyts,Kolesnikov-Livshyts,Milman,Boroczky-De}).

\medskip

As already mentioned, inequality \eqref{BM} has been generalized in many directions. It became the prototype for many similar inequalities, which bear the name of ``Brunn-Minkowski type inequalities'', where the volume functional is replaced by other functionals. Among them, we mention those verified by intrinsic volumes as functionals defined on $\K^n$, with respect to the standard Minkowski addition. Indeed, for every $k\in\{0,1,\dots,n\}$, the following inequality holds:
\begin{equation}\label{BMk}
V_k((1-t)K_0+t K_1)^{1/k}\ge (1-t)V_k(K_0)^{1/k}+t V_k(K_1)^{1/k},\quad\forall\ K_0,K_1\in\K^n,\ \forall\ t\in[0,1]
\end{equation}
(see \cite[Theorem 7.4.5]{Schneider}), where $V_k$ is the {$k$-th intrinsic volume}. In particular, when $k=1$ equality holds in the previous inequality for every $K_0, K_1$ and $t$, while for $k= n$ inequality (\ref{BMk}) is the classical Brunn-Minkowski inequality (\ref{BM}). Note that \eqref{BMk} implies a corresponding inequality with respect to the $p$-addition, for every $p\ge1$, in $\K^n_0$, due to the monotonicity of intrinsic volumes. 

\medskip

The question that we consider in this paper is whether intrinsic volumes verify a Brunn-Minkowski inequality  with respect to the $p$-addition, for $p\in[0,1)$, in $\K^n_0$. 

\medskip

To begin with, we present the case $k\in\{2,\dots,n-1\}$ (note that, for $k=0$, $V_0$ is constant, $k=n$ is the case of the volume, and the case $k=1$ will be described separately). We prove two types of results, one in the affirmative and one in the negative direction. Our first two results state the validity of the $p$-Brunn-Minkowski inequality for intrinsic volumes, in a suitable neighborhood of the unit ball $B_n$ of $\R^n$, for every $p\in[0,1)$. We start with the case $p=0$. We denote by  $\K^n_{0,s}$ the family of origin symmetric convex bodies.

\begin{theorem}\label{princ due}
Let $k\in\{2,\dots,n\}$. There exists $\eta>0$ such that if $K\in\K^n_{0,s}$ is of class $C^{2,+}$ and verifies
$$
\|1-h_K\|_{C^2(\sfe)}\le \eta,
$$
then 
\begin{equation}\label{BM-0}
V_k((1-t)\cdot B_n+_0 tK)\ge V_k(B_n)^{1-t}\ V_k(K)^{t}.
\end{equation}
Moreover, equality holds in \eqref{BM-0}, if and only if $K$ is a ball centered at the origin.
\end{theorem}

The proof of this theorem is in the same spirit of analogous (and, in fact, stronger) results concerning the volume, proved in \cite{CL,CLM,Kolesnikov-Milman}. The argument can be heuristically described as follows: \eqref{BM-0} is equivalent to the concavity of $\log(V_k)$, with respect to the $0$-addition. We compute the second variation of $\log(V_k)$, and we prove that it is negative definite at the unit ball. Then, by a continuity argument, we deduce that such second variation continues to be negative definite in a neighborhood of $B_n$. As in the case of the volume, determining the sign of the second variation amounts to analysing the spectrum of a second order elliptic operator on $\sfe$. As it is pointed out in \cite{Kolesnikov-Milman}, this method dates back to the proof of the standard Brunn-Minkowski inequality for the volume, given by Hilbert (see also \cite{Bonnesen-Fenchel}). 

Using \eqref{inclusion} and an argument based on homogeneity, from the previous result we deduce its corresponding version for $p>0$.

\begin{theorem}\label{princ_uno_nuovo} Let $k\in\{2,\dots,n\}$ and $p\in(0,1)$. There exists $\eta>0$ such that if $K\in\K^n_{0,s}$ is of class $C^{2,+}$ and verifies
$$
\|1-h_K\|_{C^2(\sfe)}\le \eta,
$$
then
\begin{equation}\label{BM-q}
V_k((1-t)\cdot B_n+_p tK)^{p/k}\ge (1-t)V_k(B_n)^{p/k}+tV_k(K)^{p/k}.
\end{equation}
Moreover, equality holds in \eqref{BM-q}, if and only if $K$ is a ball centered at the origin.
\end{theorem}

By homogeneity the same results could be stated replacing $B_n$ by the ball of radius $R>0$ centered at the origin, for an arbitrary $R$. 

In the case of the volume, $k=n$, and for $n\ge3$, Theorem \ref{princ due} is contained in \cite{CL}. Concerning Theorem \ref{princ_uno_nuovo}, we remark that when $p<1$ is sufficiently close to $1$, the validity of the $p$-Brunn-Minkowski inequality for the volume has been proved in the entire class of centrally symmetric convex bodies, with the contribution of results by Kolesnikov and Milman \cite{Kolesnikov-Milman} and Chen, Huang, Li, Liu \cite{Chen-Huang-Li-Liu} (see also \cite{Putterman}).

\medskip

We also show by counterexamples that for $p$ sufficiently close to $0$, the $p$-Brunn-Minkowski inequality for $V_k$ does not hold in $\K^n$, for any $k\in\{2,\dots,n-1\}$.

\begin{theorem}\label{BM-fails}
For every $n\ge 3$, $k\in\{2,\dots,n-1\}$, there exists $\bar p\in(0,1)$ such that for every $p\in(0,\bar p)$ the $p$-Brunn-Minkowski inequality for $V_k$ does not hold in $\K^n_{0,s}$. That is, there exist $K_0, K_1\in\K^n_{0,s}$ such that 
	\begin{equation}\label{non-BM}
	V_k(\text{\scriptsize{$\frac 12$}} \cdot K_0+_p \text{\scriptsize{$\frac 12$}} \cdot K_1)^{\frac pk} < \frac 12 V_k(K_0)^{\frac pk}+\frac 12V_k(K_1)^{\frac pk},
	\end{equation}
	if $0<p<\bar{p}$.
\end{theorem}

Indeed we built some counterexamples by considering $k$-dimensional cubes, with faces parallel to coordinate hyperplanes, embedded in $\R^n$ is such a way that the dimension of their intersection is minimized. The construction shows how the value $\bar p$ depends on $n$ and $k$. 

\medskip

The analysis of the case $k=1$ yields a reverse Brunn-Minkowski inequality: this is a direct consequence of the linearity of $V_1$ with respect to the Minkowski addition.

\begin{theorem}\label{princ tre} Let $n\ge 3$, $p\in [0,1)$, $t\in[0,1]$. For every $K_0, K_1\in\K^n_{0}$ it holds
\begin{equation}\label{BM-1}
V_1((1-t)\cdot K_0+_p t\cdot K_1)^{p}\le (1-t)V_1(K_0)^{p}+tV_1(K_1)^{p}.
\end{equation}
Moreover, equality holds if and only if either one of the two bodies $K_0$ and $K_1$ coincides with $\{0\}$, or they coincide up to a dilation. 
\end{theorem}

\medskip

As it is well known, Brunn-Minkowski type inequalities (and in particular equality conditions), are often decisive for uniqueness in the corresponding Minkowski problem. The relevant problem for intrinsic volumes is in fact the so-called Christoffel-Minkowski problem, which asks to determine a convex body when one of its area measures is prescribed (see \cite[Section 8.4]{Schneider}). As an application of Theorems \ref{princ due} and \ref{princ_uno_nuovo}, we find the following local uniqueness result for the solution of the $L_p$ version of the Christoffel-Minkowski problem, $0\le p<1$. 

\begin{theorem}\label{teo uniqueness}
Let $p\in[0,1)$ and $k\in\{2,\dots,n-2\}$. There exists $\eta>0$ such that if $K\in\K^n_{0,s}$ is of class $C^{2,+}$ and 
$$
\|1-h_K\|_{C^2(\sfe)}\le \eta,
$$
then the condition
\begin{equation}\label{eq uniqueness}
h_K^{1-p}\d S_{k-1}(K,\cdot)=\d S_{k-1}(B_n,\cdot),
\end{equation}
implies that $K=B_n$ (here $S_{k-1}$ denotes the \emph{area measure} of order $(k-1)$ of $K$). 
\end{theorem} 

\medskip

The corresponding result for $k=n$ is established in \cite{CL}. Note that \eqref{eq uniqueness} can be written as a partial differential equation on $\sfe$:
\begin{equation*}
h(x)^{1-p} S_{k-1}(h_{ij}(x)+h(x)\delta_{ij})=c_{n,k},
\end{equation*}
where $h$ indicates the support function of $ K $, $h_{ij}$, for  $i,j=1,\dots,n-1$, are the second covariant derivatives of $h$ with respect to a local orthonormal frame on $\sfe$, $\delta_{ij}$ are the usual Kronecker symbols, $S_{k-1}(h_{ij}+h\delta_{ij})$ is the elementary symmetric function of order $(k-1)$ of the eigenvalues of $h_{ij}+h\delta_{ij}$ and $c_{n,k}=\binom {n-1}{k-1}$.

\medskip

\noindent {\bf Organization of the paper.} After some preliminaries, given in Section \ref{Preliminaries}, Section \ref{section The functional} concerns properties of intrinsic volumes relevant to the computation of their first and second variations with respect to the $p$-addition, presented in Section \ref{section pv}. The proofs of Theorems  \ref{princ due}, \ref{princ_uno_nuovo} and \ref{princ tre} are given in Section \ref{section proofs}, while Theorem \ref{teo uniqueness} is proved in Section \ref{proof teo uniqueness}. Eventually, the proof of Theorem \ref{BM-fails} is contained in Section \ref{section counterexamples}.

\bigskip

\noindent {\bf Acknowledgements.} The authors are indebted to Emanuel Milman for his precious suggestions concerning Theorem \ref{princ_uno_nuovo}.  Chiara Bianchini was partially supported by the GNAMPA project ``Analisi spettrale per operatori ellittici con condizioni di Steklov o parzialmente incernierate''. Andrea Colesanti  was partially supported by the GNAMPA project ``Equazioni alle derivate parziali: aspetti geometrici, disuguaglianze collegate e forme ottime''.  Daniele Pagnini was partially supported by the GNAMPA group. Alberto Roncoroni is the holder of a postdoc funded by the INdAM and was partially supported by the GNAMPA project ``Equazioni ellittiche e disuguaglianze analitico/geometriche collegate''.

\section{Preliminaries}\label{Preliminaries}

\subsection{Notations}\label{notation}

We work in the $n$-dimensional Euclidean space $\R^n$, $n\ge2$, endowed with the Euclidean norm $\vert\cdot\vert$ and the scalar product $(\cdot,\cdot)$.  
We denote by $B_n:= \lbrace x\in\R^n : \vert x\vert \leq 1\rbrace$ and $\sfe :=\lbrace x\in\R^n : \vert x\vert = 1\rbrace$ the unit ball and the unit sphere,  respectively.


\subsubsection{Convex bodies and Wulff shapes} 
The symbol $\K^n$ indicates the set of convex bodies in $\R^n$, that is, convex and compact subsets of $\R^n$.
For every $K\in\K^n$, $h_K$ denotes the
{\em support function} of $K$, which is defined, for every  $x\in\sfe$, as:
$$
h_K(x)=\sup\{(x,y)\colon y\in K\}.
$$

We say that $K\in\K^n$ is of class $C^{2,+}$ if its boundary $\partial K$ is of class $C^2$ and the Gauss curvature is positive at every point of $\partial K$.

We denote by $\K^n_0$ the family of convex bodies containing the origin in their interior and by $\K^n_{0,s}$ the family of those elements of $\K^n_0$ which are
origin symmetric.
We underline that a convex body contains the origin if and only if its support function is non-negative on $\sfe$, and the origin is an interior point if and only if the support function is strictly positive on $\sfe$. Moreover, $K\in\K^n$ is origin symmetric if and only if $h_K$ is even.

For $p\ge 1$, the {\em $L^p$ Minkowski linear combination} of $K$ and $L$ in $\K^n_0$, with coefficients $\alpha,\beta\geq0$, denoted by
$$
\alpha\cdot K+_p\beta\cdot L,
$$
is defined through the relation
\begin{equation}\label{Lp}
h_{\alpha\cdot K+_p\beta\cdot L}=(\alpha h_K^p+\beta h_L^p)^{1/p},\quad\forall\ K,L\in\K^n_0,\;\forall\alpha,\beta\ge0.
\end{equation}
Given a continuos function $f\in C(\sfe)$ and $f>0$, we define its {\em Aleksandrov body},  or {\em Wulff shape}, as
$$
K[f]=\{x\in\R^n\colon (x,y)\le f(y),\, \forall\ y\in\sfe\}.
$$
It is not hard to prove that $K[f]$ is a convex body, and 
$$
h_{K[f]}\le f.
$$
Moreover, equality holds in the previous inequality if $f$ is a support function.
Notice that, by definition, if $f\equiv R$, where $R$ is a positive constant, it holds $K[f]=R B_n$.

For $p>0$, the {\em $L^p$ Minkowski convex combination} of  $K_0,K_1\in\K^n_0$, with parameter $t\in[0,1]$, is defined as:
\begin{equation}\label{p-somma}
(1-t)\cdot K_0+_p t\cdot K_1=K[((1-t) h_{K_0}^p+t h_{K_1}^p)^{1/p}],
\end{equation}
and we interpret the case $p=0$ in the following limiting sense:
\begin{equation}\label{0-somma}
(1-t)\cdot K_0+_0 t\cdot K_1=K[h_{K_0}^{1-t} h_{K_1}^t].
\end{equation}
Notice that if $p\ge1$ this coincides with the classical $L^p$ Minkowski linear combination defined by (\ref{Lp}).  Moreover, for all $p\in (0,1)$, the following chain of inclusions
\begin{equation}\label{inclusion_bis}
(1-t)\cdot K_0+_0 t\cdot K_1\subseteq
(1-t)\cdot K_0+_p\ t\cdot K_1\subseteq (1-t) K_0+\ t K_1
,\quad\forall\ K_0,K_1\in\K^n_0,\;\forall\ t\in[0,1]
\end{equation}
holds.  Indeed, by the monotonicity property of the $p$-means we have  
$$
h_{0}^{1-t}h_{1}^{t}\leq \big((1-t)h_{0}^p+ t h_{1}^p\Big)^{\frac 1p}\le (1-t)h_{0}+ t\, h_{1},
$$
where $h_0$ and $h_1$ denote the support functions of $K_0$ and $K_1$, respectively; hence from \eqref{0-somma} and \eqref{p-somma} we get \eqref{inclusion_bis}.

For $K\in\K^n$ and $k\in\{0,\dots,n\}$, $V_k(K)$ denotes the $k$-th intrinsic volume of $K$; the intrisic volumes are described in more details Section \ref{section The functional}.

\subsection{The matrix $Q$}
For a function $\varphi\in C^2(\sfe)$ and $i,j\in\{1,\dots,n-1\}$, we denote by $\varphi_i$, $\varphi_{ij}$ the first and second covariant derivatives with respect to a local orthonormal frame on $\sfe$. Moreover, we set
$$
Q[\varphi]=(Q_{ij}[\varphi])_{i,j=1,\dots,n-1}=(\varphi_{ij}+\varphi\delta_{ij})_{i,j=1,\dots,n-1},
$$
where $\delta_{ij}$, $i,j\in\{1,\dots,n-1\}$, denotes the Kronecker symbols. Then $Q[\varphi]$ is a symmetric matrix of order $(n-1)$.

The following proposition can be deduced, for instance, from \cite[Section 2.5]{Schneider} and it will be used several times in the paper.

\begin{proposition}\label{caratterizzazione funzioni supporto} 
Let $K\in\K^n$. The body $K$ is of class $C^{2,+}$ if and only if its support function $h_K$ is a $C^2(\sfe)$ function and 
$$
Q[h_K]>0 \quad \text{on $\sfe$}\, .
$$
\end{proposition}
We set
$$
C^{2,+}(\sfe)=\{h\in C^2(\sfe)\colon Q[h]>0\;\mbox{on}\;\sfe\};
$$
that is, $C^{2,+}(\sfe)$ is the set of support functions of convex bodies of class $C^{2,+}$. We also denote by $C^{2,+}_0(\sfe)$ the set of support functions of convex bodies of class $C^{2,+}$ in $\K^n_0$.

\subsection{Elementary symmetric functions of a matrix}\label{el_symm_functions} 

Let $N\in\N$ (in most cases we will consider $N=(n-1)$); we denote by $\Sym(N)$ the space of symmetric square matrices of order $N$. For $A=(a_{ij})\in\Sym(N)$ and for $r\in\{1,\dots,N\}$, we denote by $S_r(A)$ the \emph{$r$-th elementary symmetric function} of the eigenvalues $\lambda_1,\dots,\lambda_N$ of $A$:
$$
S_r(A)=\sum_{1\le i_1<i_2<\dots<i_r\le N}\lambda_{i_1}\dots\lambda_{i_r}.
$$
For completeness, we set $S_0(A)=1$. Note that
$$
S_N(A)=\det(A),\quad S_1(A)={\rm tr}(A).
$$
For $ r\in\{1,\ldots,N\} $ and $i,j\in\{1,\dots,N\}$, we set
$$
S_r^{ij}(A)=\frac{\partial S_r(A)}{\partial a_{ij}}.
$$
The symmetric matrix $(S_{r}^{ij}(A))_{i,j\in\{1,\dots,N\}}$ is called the {\em $r$-cofactor} matrix of $A$. In the special case $r=N$, this is the standard cofactor matrix (in particular $(S_{N}^{ij}(A))_{i,j\in\{1,\dots,N\}}=A^{-1}\det(A)$, provided $\det(A)\neq 0$). 
If $\mathrm{I}_{N}$ denotes the identity matrix of order $N$, then, for every $r\in\{1,\dots,N\}$, we have
\begin{equation}\label{Sr identity}
S_r(\mathrm{I}_{N})=\binom Nr\, .
\end{equation}
Moreover, 
$$
S^{ij}_r(\mathrm{I}_{N})= \binom{N-1}{r-1}\delta_{ij},\quad\forall\ i,j\in\{1,\dots,N\},
$$
i.e.
\begin{equation}\label{serve}
(S^{ij}_r(\mathrm{I}_{N}))_{i,j\in\{1,\dots,N\}}= \binom{N-1}{r-1} \mathrm{I}_N.
\end{equation}
This follows from \eqref{Sr identity} and from \cite[Proposition 2.1]{CS}. We also set
$$
S_r^{ij,kl}(A)=\frac{\partial^2 S_r(A)}{\partial a_{ij}\partial a_{kl}},
$$
for $ r\in\{1,\ldots,N\} $ and $ i,j,k,l\in\{1,\ldots,N\} $.

\subsection{Integration by parts formula}

The following integration by parts formula holds.

\begin{proposition}\label{ibpf} For every $h$, $\psi$, $\varphi$, $\bar\varphi\in C^2(\sfe)$, 
\begin{equation}\label{parts}
	\int_{\sfe}\bar\varphi S_k^{ij}(Q[h])(\varphi_{ij}+\varphi\delta_{ij})\, \d x=\int_{\sfe}\varphi S_k^{ij}(Q[h])(\bar\varphi_{ij}+\bar\varphi\delta_{ij})\, \d x,
\end{equation}
\begin{equation}\label{parts_bis}
\int_{\sfe}\psi S_k^{ij,kl}(Q[h])(\varphi_{ij}+\varphi\delta_{ij})(\bar\varphi_{ij}+\bar\varphi\delta_{ij})\, \d x=\int_{\sfe}\bar\varphi S_k^{ij,kl}(Q[h])(\varphi_{ij}+\varphi\delta_{ij})(\psi_{ij}+\psi\delta_{ij})\, \d x,
\end{equation}
where we have used the convention of summation over repeated indices.
\end{proposition}

The proof follows from Lemma 2.3 in \cite{CHS} (see also \cite[(11)]{CL}).

\subsection{The Poincar\'e inequality on the sphere}

Given a function $\psi\in C^1(\sfe)$, we denote by  $\nabla\psi$ the \emph{spherical gradient} of $\psi$ (that is the gradient of $\psi$ as an application from $\sfe$ to $\R$; see for instance \cite{Abate-Tovena}). 
For $\psi\in C^2(\sfe)$, we denote by $\|\psi\|_{L^2(\sfe)}, \|\nabla\psi\|_{L^2(\sfe)}, \|\psi\|_{C^2(\sfe)}$ 
the standard $L^2$ and $C^2$ norms on the sphere, respectively.

\begin{proposition}[Poincar\'e inequality on $\sfe$] For every $\psi\in C^1(\sfe)$ such that
$$
\int_{\sfe}\psi(x)\d x =0,
$$
it holds
$$
\int_{\sfe}\psi^2(x)\d x\leq \dfrac{1}{n-1} \int_{\sfe}\vert\nabla \psi(x)\vert^2\d x.
$$
\end{proposition}

The constant in the previous inequality can be improved under a symmetry assumption, as the following result shows (see, e.g., Section 2 in \cite{CL} for the proof).

\begin{proposition}[Poincar\'e inequality on $\sfe$ with symmetry] \label{P pari} Let $\psi\in C^1(\sfe)$ be even, and such that
$$
\int_{\sfe}\psi(x)\d x =0.
$$
Then
\begin{equation}\label{Poincare}
\int_{\sfe}\psi^2(x)\d x\leq \dfrac{1}{2n} \int_{\sfe}\vert\nabla \psi(x)\vert^2\d x\, ,
\end{equation}
\end{proposition}

\section{Intrinsic volumes}\label{section The functional}

Given a convex body $K\in\K^n$ of class $C^{2,+}$ and $k\in\{0,\dots,n\}$, the $k$-th intrinsic volume of $K$ can be written in the form:
$$
V_k(K)=\frac{1}{k\kappa_{n-k}}\int_{\sfe}h(x) S_{k-1}(Q[h](x))\d x,
$$
where $h=h_K\in C^{2,+}(\sfe)$ is the support function of $K$, and
$\kappa_j$ denotes the $j$-dimensional volume of the unit ball in $\R^j$ (see formulas (2.43), (4.9) and (5.56) in \cite{Schneider}). Based on the previous formula, we consider the functional
$$
F_k\colon\scdp\to[0,\infty),\quad F_k(h)=\frac{1}{k}\int_{\sfe}h(x) S_{k-1}(Q[h](x))\d x,
$$
and we define
\begin{equation}\label{caleffe cappa}
\F_k\colon\scdp\to C(\sfe),\quad \F_k(h)=S_{k-1}(Q[h]).
\end{equation}
The functionals $F_k$ and $\F_k$ have the following properties.
\begin{itemize}
\item $\displaystyle F_k(h)=\frac{1}{k}\int_{\sfe}h(x)\F_k(h)(x)\d x.
$
	\item $\F_k$ is positively homogeneous of order $(k-1)$, that is:
	$$
	\F_k(t h)=t^{k-1}\F_k(h),\quad\forall\, h\in\scdp,\;\forall t>0.
	$$
	Consequently, $F_k$ is positively homogeneous of order $k$.
	\item For every $h\in\scdp$ there exists a linear functional $L_k(h)\colon C^2(\sfe)\to C(\sfe)$ such that, for every $\varphi\in C^2(\sfe)$,
	$$
	\lim_{s\to0}\frac{\F_k(h+s\varphi)-\F_k(h)}{s}=L_k(h)\varphi,
	$$
	where $L_k(h)\varphi$ denotes $L_k(h)$ applied to $\varphi$. $L_k$ admits the following representation:
	\begin{equation}\label{CS}
	L_k(h)\varphi=S^{ij}_{k-1}(Q[h])(\varphi_{ij}+\varphi\delta_{ij}),
	\end{equation}  
	where the summation convention over repeated indices is used (see \cite[Proposition 4.2]{CS}).
	\item For every $h\in\scdp$, $L_k(h)$ is self-adjoint, that is,
	$$
	\int_{\sfe}\psi L_k(h)\varphi\d x=\int_{\sfe}\varphi L_k(h)\psi\d x,
	$$
	for every $\varphi,\psi\in C^2(\sfe)$. This follows from Proposition \ref{ibpf}.
\end{itemize}

\section{Perturbations and variations}\label{section pv}

In this section we consider a general class of functionals, with structural properties similar to those of intrinsic volumes, and in particular containing intrinsic volumes as examples. In what follows, we assume that $\K^n_0$ is endowed with the $p$-addition,  for some fixed $p\geq 0$.

Let $\FF\colon\K^n_0\to[0,\infty)$ have the following properties:
\begin{itemize}
	\item[$i)$] $\FF$ is $\alpha$-homogeneous for some $\alpha>0$:
	$$
	\FF(t\cdot K)=t^\alpha \FF(K),
	$$
	for every $K\in\K^n_0$ and for every $t>0$.
	\item[$ii)$] For every convex body of class $C^{2,+}$, $\FF$ can be expressed in the form
	$$
	\FF(K)=F(h):=\int_{\sfe} h\F(h)\d x,
	$$
	where $h$ is the support function of $K$ and $\F\colon\scdp\to C(\sfe)$. 
\end{itemize}
	
On $\F$ we will assume the same properties specified in Section \ref{section The functional}, in the case of intrinsic volumes. Namely:
	\begin{itemize}
	\item $\F$ is positively homogeneous of order $(\alpha-1)$:
	$$
	\F(t h)=t^{\alpha-1}\F(h),\quad\forall\, h\in\scdp,\;\forall t>0.
	$$
	\item $\F$ has the following differentiability property: for every $h\in\scdp$ there exists a linear functional $L(h)\colon C^2(\sfe)\to C(\sfe)$, such that for every $\varphi\in C^2(\sfe)$:
	$$
	\lim_{s\to0}\frac{\F(h+s\varphi)-\F(h)}{s}=L(h)\varphi.
	$$
	Here $L(h)\varphi$ denotes $L(h)$ applied to $\varphi$. 
	\item For every $h\in\scdp$, $L(h)$ is self-adjoint, that is, for every $\varphi,\psi\in C^2(\sfe)$,
	$$
	\int_{\sfe}\psi L(h)\varphi\d x=\int_{\sfe}\varphi L(h)\psi\d x.
	$$
\end{itemize}

We fix an element $h$ of $\scdp$ and we consider a differentiable path $h_s$ in $\scdp$, passing through $h$. In other words, for some $\varepsilon>0$, we have a map $(-\varepsilon,\varepsilon)\ni s\mapsto h_s\in\scdp$ such that
$$
h_0=h,
$$
and the following derivatives exist for every $s$ and for every $x\in\sfe$:
$$
\dot h_s(x):=\frac{d}{ds}h_s(x),\quad
\ddot h_s(x):=\frac{d^2}{ds^2}h_s(x) \quad \text{and} \quad \dddot h_s(x):=\frac{d^3}{ds^3}h_s(x).
$$
We also set
$$
\dot h=\left.\dot h_s\right|_{s=0},\quad
\ddot h=\left.\ddot h_s\right|_{s=0} \quad \text{and} \quad \dddot h=\left.\dddot h_s\right|_{s=0}.
$$
We assume that the limits giving the previous derivatives are uniform in $x$.

\medskip

Our next step is to compute the first and second derivatives of $F$ along this path, i.e., the derivatives of 
$$
(-\varepsilon,\varepsilon)\ni s\mapsto\FF(K_s)=F(h_s).
$$

\begin{proposition}\label{first variation} For every $h\in\scdp$, for every $\varphi\in C^2(\sfe)$ and for every $s\in(-\varepsilon,\varepsilon)$, with $\varepsilon>0$ sufficiently small, we have
	$$
	\frac{d}{ds}F(h_s)=\int_{\sfe}\dot h_s\F(h_s)\d x.
	$$ 
\end{proposition}

\begin{proof} We can differentiate under the integral sign, and obtain
	\begin{eqnarray*}
		\frac{d}{ds}F(h_s)&=&\frac1\alpha\int_{\sfe}\frac{d}{ds}\left[h_s\F(h_s) \right]\d x\\
		&=&\frac1\alpha\left\{\int_{\sfe}\dot h_s\F(h_s)\d x+\int_{\sfe}h_s\frac{d}{ds}\F(h_s)\d x\right\}\\
		&=&\frac1\alpha\left\{\int_{\sfe}\dot h_s\F(h_s)\d x+\int_{\sfe}h_s L(h_s)\dot h_s\d x\right\}\\
		&=&\frac1\alpha\left\{\int_{\sfe}\dot h_s\F(h_s)\d x+\int_{\sfe}\dot h_s L(h_s)h_s\d x\right\}\\
		&=&\frac1\alpha\left\{\int_{\sfe}\dot h_s\F(h_s)\d x+\int_{\sfe}\dot h_s (\alpha-1)\F(h_s)\d x\right\}\\
		&=&\int_{\sfe}\dot h_s\F(h_s)\d x.
	\end{eqnarray*}
	In this chain of equalities we have used the fact that $L$ is self-adjoint and the equality
	$$
	L(h)h=(\alpha-1)\F(h),
	$$
	which follows from the homogeneity of the functional $\F$.
\end{proof}

The next statement follows from the previous proposition, differentiating one more time under the integral.

\begin{proposition}\label{second variation} For every $h\in\scdp$, for every $\varphi\in C^2(\sfe)$ and for every $s\in(-\varepsilon,\varepsilon)$, with $\varepsilon>0$ sufficiently small, we have
	$$
	\frac{d^2}{ds^2}F(h_s)=
	\int_{\sfe}\ddot h_s\F(h_s)\d x+\int_{\sfe}\dot h_s L(h_s)\dot h_s\d x.
	$$
\end{proposition}

\section{Proof of Theorems  \ref{princ due}, \ref{princ_uno_nuovo} and \ref{princ tre}}\label{section proofs}

\subsection{The case $k\in\lbrace 2,\dots,n\rbrace$}

In this subsection we prove Theorem \ref{princ due} and Theorem \ref{princ_uno_nuovo}. 

\subsubsection{Computations and estimates of derivatives} We consider a convex body $K\in\K^n_0$ of class $C^{2,+}$, and we denote by $h\in\scdpo$ its support function. We fix $\psi\in C^2(\sfe)$ and we define, for sufficiently small $|s|$, 
\begin{equation}\label{h_s}
h_s=he^{s\psi} \, .
\end{equation}

The proof of the following result follows from \cite[Remark 3.2]{CL}.

\begin{lemma}\label{meno due due}
Let $h\in\scdpo$ and $h_s$ as in \eqref{h_s}; there exists $\eta_0>0$ (depending on $h$) with the following property: if $\psi\in C^2(\sfe)$ and
$$
\|\psi\|_{C^2(\sfe)}\le\eta_0,
$$
then
$$
h_s\in\scdpo,\quad\forall\ s\in[-2,2].
$$
\end{lemma}


In the sequel, we will always assume that $h$ and $\psi$ are such that $h_s\in\scdpo$ for every $s\in[-2,2]$.  As in Section \ref{section pv}, we denote by $\dot h_s, \ddot h_s, \dddot h_s$ the first, second and third derivatives of $h_s$ with respect to $s$, respectively. When the index $s$ is omitted, it means that these derivatives are computed at $s=0$.   

\medskip

Let $k\in\{2,\dots,n\}$; we are interested in the function
\begin{equation}\label{intrinsic_volume_def}
f_k(s):=\frac{1}{k}\int_{\sfe}h_s(x)S_{k-1}(Q[h_s](x))\d x=\frac{1}{k}\int_{\sfe}h_s(x)\F_{k}(h_s)\d x.
\end{equation}

\begin{lemma}\label{derivatives}
With the notations introduced above, we have, for every $s$:
$$
f_k'(s)=\int_{\sfe}\dot h_s(x)S_{k-1}(Q[h_s](x))\d x\, ,
$$
$$
f_k''(s)=\int_{\sfe}\ddot h_s(x)S_{k-1}(Q[h_s](x))\d x + \int_{\sfe}\dot h_s(x)S^{ij}_{k-1}(Q[h_s](x))Q_{ij}[\dot h_s](x) \d x\, , 
$$
and
$$
\begin{aligned}
f_k'''(s)=&\int_{\sfe}\dddot h_s(x)S_{k-1}(Q[h_s](x))\d x + 2\int_{\sfe}\ddot h_s(x)S^{ij}_{k-1}(Q[h_s](x))Q_{ij}[\dot h_s](x) \d x \\
&+ \int_{\sfe}\dot h_s(x)S_{k-1}^{ij,rs}(Q[h_s](x))Q_{ij}[\dot h_s](x)Q_{rs}[\dot h_s](x)\d x \\ &+\int_{\sfe}\dot h_s(x)S_{k-1}^{ij}(Q[h_s](x))Q_{ij}[\ddot h_s](x)\d x\, .
\end{aligned}
$$
\end{lemma}

\begin{proof}
The formulas for the first and the second derivatives follow from Propositions \ref{first variation} and \ref{second variation}, and from \eqref{CS} . For the third derivative, the proof is similar to the one in \cite[Lemma 3.3]{CL}:
$$
\begin{aligned}
f_k'''(s)=&\int_{\sfe}\dddot h_s(x)S_{k-1}(Q[h_s](x))\d x + \int_{\sfe}\ddot h_s(x)S^{ij}_{k-1}(Q[h_s](x))Q_{ij}[\dot h_s](x) \d x \\
&+\int_{\sfe}\ddot h_s(x)S^{ij}_{k-1}(Q[h_s](x))Q_{ij}[\dot h_s](x) \d x \\
&+ \int_{\sfe}\dot h_s(x)S_{k-1}^{ij,rs}(Q[h_s](x))Q_{ij}[\dot h_s](x)Q_{rs}[\dot h_s](x)\d x \\ &+\int_{\sfe}\dot h_s(x)S_{k-1}^{ij}(Q[h_s](x))Q_{ij}[\ddot h_s](x)\d x\, ,
\end{aligned}
$$
where we have used \eqref{CS} and \cite[Remark 3.1]{CL}.
\end{proof}

If $K$ is the unit ball, then $h\equiv 1$; consequently
\begin{equation}\label{der h_s}
h_s=e^{s\psi}\, , \quad \dot{h}_s=h_s\psi \, , \quad \ddot{h}_s=h_s\psi^2 \quad \text{and} \quad \dddot{h}_s=h_s\psi^3.
\end{equation}

These formulas, together with Lemma \ref{derivatives}, lead to the following result.

\begin{corollary}\label{Cor_derivate}
Let $K$ be the unit ball. With the notations introduced above we have
\begin{eqnarray*}
f_k(0)&=&\frac{|\sfe|}{k}\binom{n-1}{n-k},\\
f_k'(0)&=&\binom{n-1}{k-1}\int_{\sfe}\psi \d x,\\
f_k''(0)&=&\binom{n-2}{n-k}\left[\frac{(n-1)k}{k-1}\int_{\sfe}\psi^2\d x+\int_{\sfe}\psi\Delta\psi \d x\right]
\end{eqnarray*}
(where $\Delta$ denotes the spherical Laplacian).
\end{corollary}

\begin{proof}
First note that, from \eqref{h_s}, $h_0\equiv 1$.  From \eqref{intrinsic_volume_def} we get that
$$
f_k(0)=\frac{1}{k}\int_{\sfe}S_{k-1}(\mathrm{I}_{n-1})\d x=\frac{|\sfe|}{k}\binom{n-1}{k-1}.
$$
Lemma \ref{derivatives}, \eqref{serve} and \eqref{h_s} imply
$$
f_k'(0)=\binom{n-1}{k-1}\int_{\sfe}\psi \d x\, .
$$
Moreover, from Lemma \ref{derivatives}, \eqref{serve} and \eqref{der h_s} we have
\begin{eqnarray*}
f_k''(0)&=&\binom{n-1}{k-1}\int_{\sfe}\psi^2\d x+\binom{n-2}{k-2}\int_{\sfe}\psi\delta_{ij}(\psi_{ij}+\psi\delta_{ij})\d x\\
&=&\frac{(n-2)!}{(k-2)!(n-k)!}\left\{\frac{n-1}{k-1}\int_{\sfe}\psi^2\d x+\int_{\sfe}\psi\Delta\psi dx+(n-1)\int_{\sfe}\psi^2\d x\right\}\\
&=&\binom{n-2}{n-k}\left[\frac{(n-1)k}{k-1}\int_{\sfe}\psi^2\d x+\int_{\sfe}\psi\Delta\psi \d x\right]\, .
\end{eqnarray*}
\end{proof}

\begin{lemma}\label{key_lemma} Let $h\in\scdpo$, and let $\eta_0$ be as in Lemma \ref{meno due due}.  
There exists a constant $C>0$, depending on $h$, $n$ and $k$,  such that if $\psi\in C^2(\sfe)$ and 
\begin{equation}\label{hp_lemma}
\|\psi\|_{C^2(\mathbb{S}^{n-1})}\leq\eta_0, 
\end{equation}
then,  the following estimates
\begin{equation}\label{stima f_k}
\vert f_k(s)\vert\leq C  \, , \quad \text{for all $s\in[-2,2]$};
\end{equation}
\begin{equation}\label{stima f'_k}
\vert f'_k(s)\vert\leq C\Vert\psi\Vert_{C^2(\sfe)} \, , \quad \text{for all $s\in[-2,2]$};
\end{equation}
\begin{equation}\label{stima f''_k}
\vert f''_k(s)\vert\leq  C\left(\Vert\psi\Vert^2_{L^2(\sfe)}+\Vert\nabla\psi\Vert^2_{L^2(\sfe)}\right) \, , \quad \text{for all $s\in[-2,2]$};
\end{equation}
\begin{equation}\label{stima f'''_k}
\vert f'''_k(s)\vert\leq C \Vert\psi\Vert_{C^2(\sfe)}\left(\Vert\psi\Vert^2_{L^2(\sfe)}+\Vert\nabla\psi\Vert^2_{L^2(\sfe)}\right) \, , \quad \text{for all $s\in[-2,2]$};
\end{equation}
hold true.

\end{lemma}

\begin{proof}
The proof is similar to the one of Lemma 3.5 in \cite{CL}.  Throughout the proof, $C$ denotes a positive constant  depending on $h$, $n$ and $k$.

We firstly observe that, since \eqref{hp_lemma} is in force, there exists $C>0$ such that
$$
\|h_s\|_{C^2(\mathbb{S}^{n-1})}\leq C\, , \quad \text{for all $s\in[-2,2]$}.
$$
Then,
\begin{equation}\label{aggiunta}
\vert h_s(x)S_{k-1}(Q[h_s](x))\vert\leq C\, , \quad \text{for all $s\in[-2,2]$};
\end{equation}
and we immediately deduce \eqref{stima f_k}.

Now we prove \eqref{stima f'_k}. From Lemma \ref{derivatives} we get
$$
|f_k'(s)|=\left|\int_{\sfe}\dot h_s(x)S_{k-1}(Q[h_s](x))\d x\right| =\left|\int_{\sfe}\psi(x) h_s(x)S_{k-1}(Q[h_s](x))\d x\right|,
$$
and so from \eqref{aggiunta} we obtain the desired estimate \eqref{stima f'_k}. 

Let us now prove \eqref{stima f''_k}.  Again, from Lemma \ref{derivatives} and the integration by parts formula \eqref{parts}, we have 
$$
\begin{aligned}
|f_k''(s)|\leq & \left|\int_{\sfe}\psi^2(x) h_s(x)S_{k-1}(Q[h_s](x))\d x \right|+ \left|\int_{\sfe}\psi(x) h_s(x)S^{ij}_{k-1}(Q[h_s](x))Q_{ij}[\psi h_s](x) \d x\right|
\\
\leq &  C \|\psi\|^2_{L^2(\sfe)} +\left|\int_{\sfe}S^{ij}_{k-1}(Q[h_s](x))(\psi h_s)_{i}(x)(\psi h_s)_{j}(x) \d x\right|\\
\leq& C \|\psi\|^2_{L^2(\sfe)}+C\|\nabla \psi\|^2_{L^2(\sfe)},
\end{aligned}
$$
hence we have the bound in \eqref{stima f''_k}. 

Finally, we show that \eqref{stima f'''_k} holds. From Lemma \ref{derivatives} we get 
$$
\begin{aligned}
|f_k'''(s)|\leq&\left|\int_{\sfe}h_s(x)\psi^3(x)S_{k-1}(Q[h_s](x))\d x\right| \\
& + 2\left|\int_{\sfe}h_s(x)\psi^2(x)S_{k-1}^{ij}(Q[h_s](x))Q_{ij}[h_s\psi](x) \d x\right| \\
&+ \left|\int_{\sfe}h_s(x)\psi(x)S_{k-1}^{ij,rs}(Q[h_s](x))Q_{ij}[h_s\psi](x)Q_{rs}[h_s\psi](x)\d x\right| \\ &+\left|\int_{\sfe}h_s(x)\psi(x)S_{k-1}^{ij}(Q[h_s](x))Q_{ij}[h_s\psi^2](x)\d x\right|.
\end{aligned}
$$
Using formula \eqref{parts} from Proposition \ref{ibpf}, we get
$$ \int_{\sfe}h_s(x)\psi(x)S_{k-1}^{ij}(Q[h_s](x))Q_{ij}[h_s\psi^2](x)\d x=\int_{\sfe}h_s(x)\psi^2(x)S_{k-1}^{ij}(Q[h_s](x))Q_{ij}[h_s\psi](x)\d x, $$
thus
$$
\begin{aligned}
|f_k'''(s)|\leq & C \|\psi\|_{C^2(\sfe)} \left|\int_{\sfe}h_s(x)\psi^2(x)S_{k-1}(Q[h_s](x))\d x\right|\\ 
&+ C \|\psi\|_{C^2(\sfe)}  \left|\int_{\sfe}h_s(x)\psi(x)S_{k-1}^{ij}(Q[h_s](x))Q_{ij}[h_s\psi](x) \d x\right| \\ 
&+ \left|\int_{\sfe}h_s(x)\psi(x)S_{k-1}^{ij,rs}(Q[h_s](x))Q_{ij}[h_s\psi](x)Q_{rs}[h_s\psi](x)\d x\right| \\
&+\left|\int_{\sfe}h_s(x)\psi^2(x)S_{k-1}^{ij}(Q[h_s](x))Q_{ij}[h_s\psi](x)\d x\right|. 
\end{aligned}
$$
Now, arguing as we did before, we have that
$$
\begin{aligned}
|f_k'''(s)| \leq & C \|\psi\|_{C^2(\sfe)}  \|\psi\|^2_{L^2(\sfe)} + C \|\psi\|_{C^2(\sfe)}  \|\nabla \psi\|^2_{L^2(\sfe)} \\
&+ \left|\int_{\sfe}h_s(x)\psi(x)S_{k-1}^{ij,rs}(Q[h_s](x))Q_{ij}[h_s\psi](x)Q_{rs}[h_s\psi](x)\d x\right| \\
& +C\|\psi\|_{C^2(\sfe)}\|\nabla\psi\|^2_{L^2(\sfe)}  \, .
\end{aligned}
$$
The third term can be estimated, arguing as before,  in the following way:
$$
\begin{aligned}
&\left|\int_{\sfe}h_s(x)\psi(x)S_{k-1}^{ij,rs}(Q[h_s](x))Q_{ij}[h_s\psi](x)Q_{rs}[h_s\psi](x)\d x\right| \\
\leq & \left|\int_{\sfe}h_s^2(x)\psi^2(x)S_{k-1}^{ij,rs}(Q[h_s](x))Q_{rs}[h_s\psi](x)\delta_{ij}\d x\right| \\ 
&+ \left|\int_{\sfe}h_s(x)\psi(x)S_{k-1}^{ij,rs}(Q[h_s](x))Q_{rs}[h_s\psi](x)(\psi h_s)_{ij}(x)\d x\right| \\
\leq & C \|\psi\|_{C^2(\sfe)}\|\psi\|^2_{L^2(\sfe)} + C \left|\int_{\sfe}S_{k-1}^{ij,rs}(Q[h_s](x))Q_{rs}[h_s\psi](x)(\psi h_s)_{j}(x)(\psi h_s)_{i}(x)\d x\right| \\
\leq &  C \|\psi\|_{C^2(\sfe)}\|\psi\|^2_{L^2(\sfe)} + C \|\psi\|_{C^2(\sfe)}\|\nabla \psi\|^2_{L^2(\sfe)},
\end{aligned}
$$
where we used the definition of $Q_{ij}[h_s\psi_s]$ and the integration by parts formula \eqref{parts_bis}.  This concludes the proof of \eqref{stima f'''_k}, hence the proof of the lemma.

\end{proof}

\subsubsection{Proof of Theorems \ref{princ due} and \ref{princ_uno_nuovo}}

We need one last lemma.

\begin{lemma}\label{lemma concava} 
Let $h\colon\sfe\to\R$, $h\equiv1$ and let $\psi\in C^2(\sfe)$ be an even function such that
$$
\|\psi\|_{C^2(\sfe)}\le\eta_0,
$$
where $\eta_0$ is given by Lemma \ref{meno due due}. Let $f_k\colon[-2,2]\to\R$ be defined by \eqref{intrinsic_volume_def}, where $k\in\{2,\dots,n\}$. There exists a constant $\eta>0$, $\eta\le\eta_0$, such that if
\begin{equation}\label{C_2 norm}
\|\psi\|_{C^2(\mathbb{S}^{n-1})}\leq\eta\, ,
\end{equation}
then the function 
$$
s\;\mapsto\; \log f_k(s) \quad \text{is concave in $[-2,2]$}\, .
$$
Moreover, the function is strictly concave, unless $\psi$ is a constant. 
\end{lemma}

\begin{proof} We start by computing 
$$
(\log f_k)'=\dfrac{f_k'}{f_k}\, 
$$
and 
$$
(\log f_k)''=\dfrac{f_k''f_k-(f_k')^2}{f_k^2}\, .
$$
We show that
$$
H(s):=f_k(s)f_k''(s)-(f_k'(s))^2<0
$$
for all $s\in[-2,2]$,  provided $\Vert\psi\Vert_{C^2(\sfe)}\leq\eta$ and $\psi$ is not constant.
We have 
$$
H(0)=f_k(0)f_k''(0)-(f_k')^2(0)\, .
$$
First we assume that  
\begin{equation}\label{media_nulla}
\int_{\sfe}\psi\d x=0 \, .
\end{equation}
According to Corollary \ref{Cor_derivate}, we have that
$$
f_k(0)=\frac{|\sfe|}{k}\binom{n-1}{n-k}
$$
and 
$$
f_k'(0)=\binom{n-1}{k-1}\int_{\sfe}\psi\d x= 0\, . 
$$
Moreover,
\begin{align*}
f_k''(0)&=\binom{n-2}{n-k}\left[\frac{(n-1)k}{k-1}\int_{\sfe}\psi^2\d x+\int_{\sfe}\psi\Delta\psi \d x\right]\\
&=\binom{n-2}{n-k}\left[\frac{(n-1)k}{k-1}\int_{\sfe}\psi^2\d x-\int_{\sfe}\|\nabla\psi\|^2 \d x\right]
\end{align*}
Now,  since \eqref{media_nulla} is in force and $\psi$ is even, from Proposition \ref{P pari} we get 
\begin{equation*}
f_k''(0)\leq\binom{n-2}{n-k}\left[\frac{(n-1)k}{2n(k-1)}-1\right]\int_{\sfe}\|\nabla\psi\|^2 \d x\, .
\end{equation*}
As 
$$
\frac{(n-1)k}{2n(k-1)}<1\, , 
$$
we may write
$$
H(0)\leq-\gamma\,\|\nabla\psi\|_{L^2(\sfe)}^2,
$$
where $\gamma>0$ depends on $n$ and $k$. Now,  for every $s\in [-2,2]$ there exists $\bar{s}$ between $0$ and $s$ such that
$$
H(s)=H(0)+sH'(\bar{s})=H(0)+s\left[f_k(\bar{s})f'''_k(\bar{s})-f'_k(\bar{s})f''_k(\bar{s})\right]\, ;
$$
from Lemma \ref{key_lemma},  we know that \eqref{stima f_k}, \eqref{stima f'_k}, \eqref{stima f''_k} and \eqref{stima f'''_k} hold true,  hence 
$$
\vert sH'(\bar{s})\vert\leq C\eta \left(\Vert\psi\Vert^2_{L^2(\sfe)}+\Vert\nabla\psi\Vert^2_{L^2(\sfe)}\right)
\le C\eta \|\nabla\psi\|_{L^2(\sfe)}^2\,,
$$
where we used Proposition \ref{P pari} again. We have then proved the concavity of $f$, and the strict concavity whenever $\|\nabla\psi\|_{L^2(\sfe)}>0$, i.e. whenever $\psi$ is not a constant. 

Now we drop the assumption \eqref{media_nulla}.  Given $\psi\in C^2(\sfe)$,  let
\begin{equation}\label{funzione_media_nulla}
m_\psi=\frac1{|\sfe|}\int_{\sfe}\psi \d x\quad\mbox{and}\quad
\bar\psi=\psi-m_\psi.
\end{equation}
Clearly $\bar{\psi}\in C^2(\sfe)$ and $\bar{\psi}$ verifies \eqref{media_nulla}. Moreover
$$
\|\bar{\psi}\|_{C^2(\mathbb{S}^{n-1})}\leq \|\psi\|_{C^2(\mathbb{S}^{n-1})} + \vert m_\psi \vert\leq 2 \|\psi\|_{C^2(\mathbb{S}^{n-1})}   \, ,
$$
hence if $ \|\psi\|_{C^2(\mathbb{S}^{n-1})}\leq \eta_0/2$ then $\bar{\psi}$ satisfies \eqref{C_2 norm}. Since 
$$
\bar{h}_s:=e^{s\bar{\psi}}=e^{s(\psi-m_\psi)}=e^{-sm_\psi} h_s,
$$
we have
$$
Q[\bar{h}_s]=e^{-sm_\psi} Q[h_s] \quad \text{and} \quad S_{k-1}(Q[\bar{h}_s](x))=e^{-(k-1)sm_\psi}S_{k-1}(Q[h_s](x)),
$$
thus
\begin{equation}\label{intrinsic_volume}
\bar{f}_k(s):=\frac{1}{k}\int_{\sfe}\bar{h}_s(x)S_{k-1}(Q[\bar{h}_s](x))\d x=e^{-ksm_\psi}f_k(s)\, .
\end{equation}
We conclude that $\log \bar{f}_k$ and $\log f_k$ differ by a linear term, and the concavity (resp. the strict concavity) of $\bar{f}_k$ is equivalent to that of $f_k$.  On the other hand, by the first part of the proof $\log \bar{f}_k$ is concave (and strictly concave unless $\psi$ is constant), as long as $\|\bar{\psi}\|_{C^2(\mathbb{S}^{n-1})}$ is sufficiently small, and this condition is verified if $\|\psi\|_{C^2(\mathbb{S}^{n-1})}$ is sufficiently small.  

Finally, note that, by \eqref{h_s}, if $\psi=\psi_0$ is constant then $h_s=e^{\psi_0 s}$. Consequently 
$$
f_k(s)=c e^{k\psi_0 s},\quad c>0,
$$ 
whence $\log f_k(s)$ is linear. 

The proof of the lemma is complete.
\end{proof}

\begin{proof}[Proof of Theorem \ref{princ due}.]
Let $\eta>0$ be as in Lemma \ref{lemma concava}, and let $K\in \K^n_{0,s}$ be of class $C^{2,+}$ and such that 
\begin{equation}
\|1-h\|_{C^2(\sfe)}\le\eta,
\end{equation}
where $h$ is the support function of $K$.  This implies that $h>0$ on $\sfe$, and therefore we can set $\psi=\log h\in C^2(\sfe)$; thus we may write $h$ in the form $h=e^\psi$. Define, for $t\in[0,1]$,
$$
K_t=(1-t)\cdot B_n+_0 t\cdot K,
$$
and let $h_t$ be the support function of $K_t$; then
$$
h_t=1^{1-t}h^t=e^{t\psi}.
$$
Hence $V_k(K_t)$ is concave in $[-2,2]$, which proves \eqref{BM-0}. Moreover, $V_k(K_t)$ is strictly concave unless $\psi$ is constant, and the latter condition is equivalent to say that $h$ is constant, i.e. $K$ is a ball centered at the origin. 
\end{proof}

The following remark will be useful for the proof of Theorem \ref{princ_uno_nuovo}.

\begin{remark}\label{rmkBM0-scaling}
	{If two convex bodies $K_0$ and $K_1$ satisfy the log-Brunn-Minkowski inequality
		\begin{equation}\label{foot2}
		V_k((1-t)\cdot K_0+_0 t\cdot K_1)\ge V_k(K_0)^{1-t}\ V_k(K_1)^{t},  \quad \text{for all $t\in [0,1]$}\,, 
		\end{equation}
		then $\alpha K_0$ and $\beta K_1$ satisfy the same log-Brunn-Minkowski inequality, for $\alpha, \beta>0$.
		
		Indeed, we consider the convex bodies $\alpha K_0$ and $\beta K_1$; from the definition of $0$-sum we have
		\begin{equation}\label{foot1}
		(1-t)\cdot \alpha K_0+_0 t\cdot \beta K_1=K[(\alpha h_0)^{1-t}(\beta h_1)^{t}]=\alpha^{1-t} \beta^{t} K[ h_0^{1-t} h_1^{t}]=\alpha^{1-t} \beta^{t} [(1-t)\cdot  K_0+_0 t\cdot K_1] \, ,
		\end{equation}
		where $h_0$ and $h_1$ denote the support functions of $K_0$ and $K_1$, respectively.  Hence, from the fact that $V_k$ is $k$-homogeneous, \eqref{foot1} and \eqref{foot2} we obtain   
		\begin{align*}
		V_k((1-t)\cdot \alpha K_0+_0 t\cdot \beta K_1)&=\alpha^{(1-t)k} \beta^{tk} V_k((1-t)\cdot  K_0+_0 t\cdot K_1) \\ 
		&\geq \alpha^{(1-t)k} \beta^{tk}V_k(K_0)^{1-t}\ V_k(K_1)^{t} \\
		&=  (\alpha^{k} V_k(K_0))^{1-t}\   (\beta^{k} V_k(K_1))^{t} \\
		&= V_k(\alpha K_0)^{1-t}\    V_k(\beta K_1)^{t} \, ,
		\end{align*}
		i.e.  $\alpha K_0$ and $\beta K_1$ satisfy the log-Brunn-Minkowski inequality \eqref{foot2} too. By the previous argument, it is clear that we have equality in the inequality for $K_0$ and $K_1$ if and only if we have equality in the inequality for $\alpha K_0$ and $\beta K_1$.
	}
\end{remark}

\begin{proof}[Proof of Theorem \ref{princ_uno_nuovo}]

Let $B_n$ and $K$ be as in the statement of Theorem \ref{princ_uno_nuovo}. We define 
$$
\tilde{B_n}:=\dfrac{1}{V_k(B_n)^{1/k}} B_n \quad \text{and} \quad \tilde{K}:=\dfrac{1}{V_k(K)^{1/k}} K \, ;
$$
observe that
\begin{equation}\label{volum=1}
V_k(\tilde{B_n})=1=V_k(\tilde{K})\, ,
\end{equation}
since $V_k$ is $k$-homogeneous.  Because of Theorem \ref{princ due}, $B_n$ and $K$ satisfy the log-Brunn-Minkowski inequality
$$
V_k((1-t)\cdot B_n+_0 t\cdot K)\ge V_k(B_n)^{1-t}\ V_k(K)^{t}\, , \quad \text{for all $t\in [0,1]$}\,.
$$
Thanks to Remark \ref{rmkBM0-scaling}, $\tilde{B_n}$ and $\tilde{K}$ satisfy the same log-Brunn-Minkowski inequality:
\begin{equation}\label{BM_tilde}
V_k((1-t)\cdot \tilde{B_n}+_0 t\cdot \tilde{K})\ge V_k(\tilde{B_n})^{1-t}\ V_k(\tilde{K})^{t}\, , \quad \text{for all $t\in [0,1]$}\,.
\end{equation}
Now,  for every $t\in [0,1]$ we define 
$$
\tilde{t}:=\dfrac{t V_k(K)^{p/k}}{(1-t)V_k(B_n)^{p/k}+tV_k(K)^{p/k}}\, .
$$
Clearly $\tilde{t}\in [0,1]$ and 
$$
1-\tilde{t}=\dfrac{(1-t) V_k(B_n)^{p/k}}{(1-t)V_k(B_n)^{p/k}+tV_k(K)^{p/k}}\, ,
$$
hence, applying \eqref{BM_tilde} with $t=\tilde{t}$, we find
$$
V_k((1-\tilde{t})\cdot \tilde{B_n}+_0 \tilde{t}\cdot \tilde{K})\ge V_k(\tilde{B_n})^{1-\tilde{t}}\ V_k(\tilde{K})^{\tilde{t}}=1\, , 
$$
where we have used \eqref{volum=1}.  On the other hand,
\begin{align*}
&V_k((1-\tilde{t})\cdot \tilde{B_n}+_0 \tilde{t}\cdot \tilde{K}) \\ 
&=V_k\left(\dfrac{(1-t) V_k(B_n)^{p/k}}{(1-t)V_k(B_n)^{p/k}+tV_k(K)^{p/k}} \cdot \dfrac{1}{V_k(B_n)^{1/k}} B_n +_0 \dfrac{t V_k(K)^{p/k}}{(1-t)V_k(B_n)^{p/k}+tV_k(K)^{p/k}} \cdot \dfrac{1}{V_k(K)^{1/k}} K\right) \\
& =V_k\left(\dfrac{(1-t) V_k(B_n)^{p/k}}{(1-t)V_k(B_n)^{p/k}+tV_k(K)^{p/k}}  \dfrac{1}{V_k(B_n)^{p/k}}\cdot B_n +_0 \dfrac{t V_k(K)^{p/k}}{(1-t)V_k(B_n)^{p/k}+tV_k(K)^{p/k}} \dfrac{1}{V_k(K)^{p/k}} \cdot K\right) \\
&= V_k\left(\dfrac{(1-t)}{(1-t)V_k(B_n)^{p/k}+tV_k(K)^{p/k}} \cdot B_n +_0 \dfrac{t}{(1-t)V_k(B_n)^{p/k}+tV_k(K)^{p/k}} \cdot K\right) \\
&=\dfrac{1}{((1-t)V_k(B_n)^{p/k}+tV_k(K)^{p/k})^{k/p}}V_k\left((1-t) \cdot B_n +_0 t \cdot K\right),
\end{align*}
where we used the fact that 
$$
V_k(\lambda\cdot K)=V_k(\lambda^{1/p}K)=\lambda^{k/p}V_k(K), \quad \text{for all $\lambda>0$} \, .
$$
Summing up, we have that 
$$
V_k\left((1-t) \cdot B_n +_0 t \cdot K\right) \geq ((1-t)V_k(B_n)^{p/k}+tV_k(K)^{p/k})^{k/p}\, .
$$
The conclusion now follows from the inclusion \eqref{inclusion_bis},
which gives
\begin{equation}\label{last}
V_k\left((1-t) \cdot B_n +_p t \cdot K\right) \geq V_k\left((1-t) \cdot B_n +_0 t \cdot K\right) \geq ((1-t)V_k(B_n)^{p/k}+tV_k(K)^{p/k})^{k/p} \, .
\end{equation}

Assuming that equality holds in \eqref{last}, we see that we have equality in \eqref{BM_tilde} as well. By Remark \ref{rmkBM0-scaling} and by the discussion of equality conditions int Theorem \ref{princ due}, we obtain that $K$ has to nomothetic to $B_n$. The vice versa of this statement follows from homogeneity.
\end{proof}

\subsection{The case $k=1$}

In this subsection we prove Theorem \ref{princ tre}. 

\begin{proof}[Proof of Theorem \ref{princ tre}]
Let $K_0, K_1\in\K^n_{0}$ and $p\in [0,1)$.
For all $t\in [0,1]$, 
$$
V_1((1-t)\cdot K_0+_p t\cdot K_1)\leq V_1((1-t) K_0+ t K_1)= (1-t)V_1(K_0)+ t V_1 (K_1) \, ,
$$
where we have used \eqref{inclusion_bis} and the fact that $V_1(\cdot)$ is linear.  So, we get
\begin{equation}\label{BM-max}
V_1((1-t)\cdot K_0+_p t\cdot K_1)\leq \max \{ V_1(K_0), V_1(K_0)\}\, , \quad \text{for all $t\in [0,1]$} \, .
\end{equation}
The argument needed in order to obtain \eqref{BM-1} from \eqref{BM-max} is standard, but for the sake of completeness we report it. First, we may assume that the dimension of $K_0$ and $K_1$ is at least one. Indeed, if at least one of them has only one point, then it coincides with $\{0\}$ and the inequality that we want to prove follows immediately. 

We consider, similarly to what we did before,
$$
\tilde{K_0}:=\dfrac{1}{V_1(K_0)}K_0 \quad \text{and} \quad \tilde{K_1}:=\dfrac{1}{V_1(K_1)}K_1 \, , 
$$
and we observe that 
$$
V_1(\tilde{K_0})=1=V_1(\tilde{K_1})\, .
$$
Moreover,  for every $t\in [0,1]$ we define 
$$
\tilde{t}:=\dfrac{t V_1(K_1)^{p}}{(1-t)V_1(K_0)^{p}+tV_1(K_1)^{p}}\, .
$$
Applying \eqref{BM-max} with $K_0=\tilde{K_0}$,  $K_1=\tilde{K_1}$ and $t=\tilde{t}$, we obtain 
$$
V_1((1-\tilde{t})\cdot \tilde{K_0}+_p \tilde{t}\cdot \tilde{K_1})\leq 1 \, .
$$
Inequality \eqref{BM-1} follows from homogeneity.

Let us now discuss equality conditions. It is straightforward to verify that if $K_0$ and $K_1$ coincide up to a dilation, or at least one of them coincides with $\{0\}$, we have equality. Vice versa, assume that equality holds, for some $K_0$, $K_1$ and $t$, and that both $K_0$ and $K_1$ are not the set $\{0\}$. Then $V_1(K_0)>0$ and $V_1(K_1)>0$, and by the same procedure used in the first part of this proof, we may find $\tilde K_0, \tilde K_1\in\K^n_0$ and $\tilde t\in[0,1]$ such that
\begin{eqnarray*}
&&\tilde K_0=\alpha_0 K_0,\quad \tilde K_1=\alpha_1 K_1, \quad\mbox{with $\alpha_0,\alpha_1>0$},\\
&&V_1(\tilde K_0)=V_1(\tilde K_1)=1,\quad V_1((1-\tilde t)\cdot\tilde K_0+_p\tilde t\cdot\tilde K_1)=1.
\end{eqnarray*}
Then
\begin{eqnarray*}
1=V_1((1-\tilde t)\cdot\tilde K_0+_p\tilde t\cdot\tilde K_1)
\le V_1((1-\tilde t)\tilde K_0+\tilde t\tilde K_1)=(1-\tilde t)V_1(\tilde K_0)+\tilde tV_1(\tilde K_1)=1.
\end{eqnarray*}
In particular
$$
V_1((1-\tilde t)\cdot\tilde K_0+_p\tilde t\cdot\tilde K_1)
=V_1((1-\tilde t)\tilde K_0+\tilde t\tilde K_1),
$$
and, thanks to \eqref{inclusion_bis}, we also have the inclusion
$$
(1-\tilde t)\cdot\tilde K_0+_p\tilde t\cdot\tilde K_1\subseteq(1-\tilde t)\tilde K_0+\tilde t\tilde K_1.
$$
By the monotonicity of $V_1$, this implies that 
$$
(1-\tilde t)\cdot\tilde K_0+_p\tilde t\cdot\tilde K_1=(1-\tilde t)\tilde K_0+\tilde t\tilde K_1,
$$
and therefore
$$
\left[(1-\tilde t)h_{\tilde K_0}^p+\tilde t h_{\tilde K_1}^p\right]^{1/p}=
(1-\tilde t)h_{\tilde K_0}+\tilde t h_{\tilde K_1}
$$
(here we are assuming $0<p<1$, but the argument for $p=0$ needs only minor modifications). From the characterization of equality conditions in the inequality between the $p$-mean and the arithmetic mean, we deduce that $h_{\tilde K_0}=h_{\tilde K_1}$. This implies that $\tilde K_0=\tilde K_1$, and then $K_0$ and $K_1$ have to be nomothetic. 
\end{proof}

%

\section{Local uniqueness for the $L_p$ Christoffel-Minkowski problem}\label{proof teo uniqueness}

In this section, we prove Theorem \ref{teo uniqueness}. We will need the following preliminary result.

\begin{lemma}\label{non si voleva} Let $p\in(0,1)$. There exists $\eta>0$ such that if $K\in\K^n_{0,s}$ is of class $C^{2,+}$ and 
$$
\|1-h_K\|_{C^2(\sfe)}\le \eta,
$$
then for every $t\in[0,1]$ the function $h_t\colon\sfe\to\R$ defined by
$$
h_t=[(1-t)+t h_K^p]^{1/p}
$$ 
is the support function of a convex body $K_t\in\K^n_{0,s}$, of class $C^{2,+}$. 
\end{lemma}

\begin{proof} Note that $h_t\in C^2(\sfe)$ for every $t\in[0,1]$. By Proposition \ref{caratterizzazione funzioni supporto}, we need to prove that 
$$
Q[h_t]>0\quad\mbox{on $\sfe$}
$$
for every $t\in[0,1]$. By contradiction, assume that there exist a sequence $\eta_j>0$, $j\in\N$, converging to $0$, a sequence of convex bodies $K_j\in\K^n_{0,s}$, of class $C^{2,+}$, a sequence $t_j\in[0,1]$ and a sequence $x_j\in\sfe$, such that, denoting by $h_j$ the support function of $K_j$,
$$
\|1-h_j\|_{C^2(\sfe)}\le \eta_j,
$$
and
$$
Q[h_{j}(x_j)]\le 0.
$$ 
Clearly $h_j$ converges to the constant function $h_0\equiv1$ in $C^2(\sfe)$, and hence $h_{j}$ converges to $h_0\equiv 1$ in $C^2(\sfe)$. Up to subsequences, we may also assume that $t_j$ and $x_j$ converge to $\bar t\in[0,1]$ and $\bar x\in\sfe$, respectively. As a consequence of these facts, by the continuity of $Q$ we get
$$
Q[h_0(\bar x)]\le 0,
$$
which is a contradiction, as $Q[h_0]$ is the identity matrix.
\end{proof}

\begin{proof}[Proof of Theorem \ref{teo uniqueness}] We first consider the case $p>0$. Let $\bar\eta>0$ be smaller than the two positive quantities, both called $\eta$, appearing in Theorem \ref{princ_uno_nuovo} and Lemma \ref{non si voleva}. Let $K\in\K^n_{0,s}$ be of class $C^{2,+}$ and such that
$$
\|1-h_K\|_{C^2(\sfe)}\le \eta.
$$
Up to replacing $\eta$ with a smaller constant, we may assume that $h_K>0$ on $\sfe$. For simplicity, in the rest of the proof we will write $h$ instead of $h_K$. We also set
$$
K_t=(1-t)\cdot B_n+_p t\cdot K,\quad\forall\, t\in[0,1].
$$
By the definition of $p$ addition and Lemma \ref{non si voleva}, the support function $h_t$ of $K_t$ is given by:
$$
h_t=\left[(1-t)+th^p\right]^{1/p}.
$$
We now consider the functions $f,g\colon[0,1]\to\R$ defined by:
\begin{eqnarray*}
&&f(t)=[V_k(K_t)]^{p/k}=\left[
\frac 1{k\kappa_{n-k}}\int_{\sfe} h_t S_{k-1}(Q[h_t])\d x
\right]^{p/k},\\
&&g(t)=(1-t)[V_k(B_n)]^{p/k}+t[V_k(K)]^{p/k}.
\end{eqnarray*}
By Theorem \ref{princ_uno_nuovo}, 
$$
f(t)\ge g(t), \quad \forall\ t\in[0,1].
$$
Moreover $f(0)=g(0), f(1)=g(1)$, so that
\begin{equation}\label{confronto derivate}
f'(0)\ge g'(0),\quad f'(1)\le g'(1).
\end{equation}
Note that
\begin{equation}\label{derivate g}
g'(0)=g'(1)=[V_k(K)]^{p/k}-[V_k(B_n)]^{p/k}.
\end{equation}
By Lemma \ref{derivatives}, we have:
\begin{eqnarray}\label{catena}
f'(0)&=&\frac pk V_k(B_n)^{\frac{p}{k}-1}\left[\frac1{p\kappa_{n-k}}\int_{\sfe}(h^p-1)S_{k-1}(Q[h_0])\d x\right]\\
&=&[V_k(B_n)]^{\frac{p}{k}-1}\frac1{k\kappa_{n-k}}\int_{\sfe}h^p S_{k-1}(Q[h_0])\d x
-[V_k(B_n)]^{\frac{p}{k}-1}\frac1{k\kappa_{n-k}}\int_{\sfe} S_{k-1}(Q[h_0])\d x\nonumber\\
&=&[V_k(B_n)]^{\frac{p}{k}-1}\frac1{k\kappa_{n-k}}\int_{\sfe}h S_{k-1}(Q[h])\d x-[V_k(B_n)]^{\frac{p}{k}}\nonumber\\
&=&[V_k(B_n)]^{\frac{p}{k}-1}[V_k(K)]-[V_k(B_n)]^{\frac{p}{k}},\nonumber
\end{eqnarray}
where we have used \eqref{eq uniqueness}. From \eqref{confronto derivate}, \eqref{derivate g} and \eqref{catena}, we get
$$
V_k(K)\ge V_k(B_n).
$$
In a similar way, from the comparison $g'(1)\ge f'(1)$ we obtain the reverse inequality. Hence $V_k(K)=V_k(B_n)$, which implies that $f'(0)=f'(1)=0$. As $f$ is concave, we have that $f$ is constant in $[0,1]$, which means that the inequality \eqref{BM-q} becomes an equality for $K$. By Theorem \ref{princ_uno_nuovo}, $K$ is a dilation of $B_n$. On the other hand, \eqref{eq uniqueness} implies $K=B_n$.   

The proof in the case $p=0$ is similar; Theorem \ref{princ due} and Lemma \ref{meno due due} will have to be used instead of Theorem \ref{princ_uno_nuovo} and Lemma \ref{non si voleva}, respectively. 
\end{proof}

\section{Proof of Theorem \ref{BM-fails}: counterexamples}\label{section counterexamples}

In this section we show that for every $k\in\{2,...,n-1\}$ there exists $\bar{p}$ such that the $p$-Brunn-Minkowski inequality for the intrinsic volumes does not hold, for every $p<\bar{p}$, that is, (\ref{non-BM})
is satisfied for suitable $ K_0,K_1\in\mathcal{K}_{0,s}^n $.

Given $k\in\{ 2,\dots,n-1 \}$, we consider 
$$
K_0:=\{ (x_1,\dots,x_n)\in\mathbb{R}^n \, : \, x_j=0 \, , \, \forall \, j=1,\dots,n-k \, \text{ and } \, |x_i|\leq 1 \, , \, \forall \, i=n-k+1,\dots,n \} \, ,
$$
and
$$
K_1:=\{ (x_1,\dots,x_n)\in\mathbb{R}^n \, : \,|x_i|\leq 1 \, , \, \forall \, i=1,\dots,k \, \text{ and } \, x_j=0 \, , \, \forall \, j=k+1,\dots,n \} \, .
$$
$K_0$ and $K_1$ are $k$-dimensional cubes of side length $2$; therefore $V_k(K_0)=V_k(K_1)=2^k$. We set, for $p\in(0,1)$,
$$
K_p:=\dfrac{1}{2}\cdot K_0+_p\dfrac{1}{2}\cdot K_1,
$$
so that its support function $h_{K_p}$ satisfies
\begin{equation}\label{support_controes}
h_{K_p}(x):=h_{K[(\frac{1}{2}h_{K_0}^p+\frac{1}{2}h_{K_1}^p)^{1/p}]}(x)\leq \left(\frac{1}{2}h_{K_0}^p(x)+\frac{1}{2}h_{K_1}^p(x)\right)^{1/p} \quad \text{for all $x\in\mathbb{R}^n$}.
\end{equation}
We denote by $\{e_1,\dots,e_n\}$ the standard orthonormal basis of $\R^n$, and we treat the cases $k>n/2$ and $k\leq n/2$ separately. 

\medskip

\noindent
$\bullet$ \emph{Case $k>n/2$.} In this case we have 
$$
h_{K_0}(\pm e_i)=\begin{cases} 0 &\mbox{if } i\in\{ 1,\dots,n-k \} \\
1 & \mbox{if } i\in\{ n-k+1,\dots,k\} \\
1 & \mbox{if } i\in\{ k+1,\dots,n\}, \end{cases}\quad\mbox{and}
\quad h_{K_1}(\pm e_i)=\begin{cases} 1 &\mbox{if } i\in\{ 1,\dots,n-k \} \\
1 & \mbox{if } i\in\{ n-k+1,\dots,k\} \\
0 & \mbox{if } i\in\{ k+1,\dots,n\}. \end{cases}
$$
Therefore, from \eqref{support_controes},
$$
h_{K_p}(\pm e_i)\leq \begin{cases} 2^{-1/p} &\mbox{if } i\in\{ 1,\dots,n-k \} \\
1 & \mbox{if } i\in\{ n-k+1,\dots,k\} \\
2^{-1/p} & \mbox{if } i\in\{ k+1,\dots,n\}. \end{cases}
$$
We deduce that 
$$
K_p\subseteq K:=\left[-2^{-1/p},2^{-1/p} \right]^{n-k}\times[-1,1]^{2k-n}\times\left[-2^{-1/p},2^{-1/p} \right]^{n-k}\, ,
$$
where $[x_0,y_0]^m$ indicates the $m$-dimensional cube given by the product of $ m $ copies of $[x_0,y_0]$. This implies
\begin{eqnarray*}
V_k(K_p)\leq V_k(K)&=&V_k\left(\left[-2^{-1/p},2^{-1/p} \right]^{2n-2k}\times[-1,1]^{2k-n}\right)\\
&=&V_k\left(\prod_{i=1}^n[-a_i,a_i]\right)=2^k\sum_{1\leq i_1<\ldots<i_k\leq n}a_{i_1}\ldots a_{i_k},
\end{eqnarray*}
where
$$
a_i=\begin{cases} 2^{-1/p} &\mbox{if } i\in\{ 1,\dots,2n-2k \} \\
1 & \mbox{if } i\in\{ 2n-2k+1,\dots,n\}.
\end{cases}
$$
Notice that, since $k<n$, we have $n-(2n-2k)=2k-n<k$, hence when choosing $k$ intervals among $\{[-a_i,a_i]\}_{i=1,...,n}$, at least one of the them is of the form 
$ \left[-2^{-1/p},2^{-1/p} \right]$.  We are going to discuss separately the cases $ k\leq\frac{2}{3}n $ and $ k>\frac{2}{3}n $. 

If $ k\leq\frac{2}{3}n $, then $ 2n-2k\geq k $,
and
$$
V_k(K_p)\leq 2^k\sum_{1\leq i_1<\ldots<i_k\leq n}a_{i_1}\ldots a_{i_k}=2^k\sum_{i=1}^k\binom{2n-2k}{i}2^{-i/p}\leq 2^{k-1/p}\sum_{i=1}^k\binom{2n-2k}{i}=:C_{n,k}\; 2^{k-1/p},
$$
whereas if $ k>\frac{2}{3}n $, then $ 2n-2k<k $ and
\begin{eqnarray*}
V_k(K_p)&\leq& 2^k\sum_{1\leq i_1<\ldots<i_k\leq n}a_{i_1}\ldots a_{i_k}=2^k\sum_{i=1}^{2n-2k}\binom{2n-2k}{i}2^{-i/p}\\
&\leq& 2^{k-1/p}\sum_{i=1}^{2n-2k}\binom{2n-2k}{i}=2^{k-1/p}(2^{2n-2k}-1).
\end{eqnarray*}

Since $V_k(K_0)=V_k(K_1)=2^k$, we have
\begin{equation*}
\left(\frac{1}{2}V_k(K_0)^\frac{p}{k}+\frac{1}{2}V_k(K_1)^\frac{p}{k}\right)^\frac{k}{p}=2^k,
\end{equation*}
while
\begin{equation*}
V_k(K_p)\leq\begin{cases}C_{n,k}\; 2^{k-1/p}&\mbox{if }\frac{n}{2}<k\leq\frac{2}{3}n\\
 2^{k-1/p}(2^{2n-2k}-1)&\mbox{if }\frac{2}{3}n<k\leq n-1\end{cases}
\end{equation*}

If $ \frac{n}{2}<k\leq\frac{2}{3}n $, consider
$$
\bar{p}=\dfrac{1}{\log_2(C_{n,k})},
$$
and let $p<\bar{p}$ (note that $C_{n,k}>1$, so that $\bar p>0$). Hence
$$
2^k > C_{n,k}\; 2^{k-1/p},
$$
that is, the $p$-Brunn-Minkowski inequality fails.

If $ \frac{2}{3}n<k\leq n-1 $, we choose
$$
\bar{p}= \frac{1}{\log_2\left(2^{2n-2k}-1\right)}
$$
and we take $p<\bar{p}$.
Hence
$$
2^k > 2^{k-\frac{1}{p}}(2^{2n-2k}-1),
$$
that is, the $p$-Brunn-Minkowski inequality fails.

\medskip 

$\bullet$ \emph{Case $k\leq n/2$.}  In this case we have 
$$
h_{K_0}(\pm e_i)=\begin{cases} 0 &\mbox{if } i\in\{ 1,\dots,k \} \\
0 & \mbox{if } i\in\{k+1,\dots,n-k\} \\
1 & \mbox{if } i\in\{ n-k+1,\dots,n,\} \end{cases}\quad\mbox{and}\quad
h_{K_1}(\pm e_i)=\begin{cases} 1 &\mbox{if } i\in\{ 1,\dots,k \} \\
0 & \mbox{if } i\in\{ k+1,\dots,n-k\} \\
0 & \mbox{if } i\in\{ n-k+1,\dots,n.\} \end{cases}
$$
Consequently, from \eqref{support_controes},
$$
h_{K_p}(\pm e_i)\leq \begin{cases} 2^{-1/p} &\mbox{if } i\in\{ 1,\dots,k \} \\
0 & \mbox{if } i\in\{ k+1,\dots,n-k\} \\
2^{-1/p} & \mbox{if } i\in\{ n-k+1,\dots,n\} \end{cases}
$$
and we deduce that 
$$
K_p\subseteq K:=\left[-2^{-1/p},2^{-1/p} \right]^k\times\{0\}^{n-2k}\times\left[-2^{-1/p},2^{-1/p} \right]^k\, .
$$
Therefore,
$$
V_k(K_p)\leq V_k\left(\left[-2^{-1/p},2^{-1/p}\right]^{2k}\times\{0\}^{n-2k}\right)=\binom{2k}{k}2^{k-\frac{k}{p}}.
$$
Let $\bar{p}=\frac{k}{\log_2{\binom{2k}{k}}}$ and consider $p<\bar{p}$. We have
	$$
	\binom{2k}{k} < 2^{\frac{k}{p}},
	$$
which is equivalent to $2^k >\binom{2k}{k} 2^{k-\frac{k}{p}}$, which entails that the $p$-Brunn-Minkowski inequality fails.

Summing things up, we have the following result: for every fixed $k\in\{2,...,n-1\}$, let
$$
\bar{p}_k= 
\begin{cases}\displaystyle{\frac{k}{\log_2\binom{2k}{k}}}&\mbox{for }1< k\leq\frac{n}{2},\\
                                 \medskip
                                 \displaystyle{\frac{1}{\log_2\sum_{i=1}^k\binom{2(n-k)}{i}}}&\mbox{for }\frac{n}{2}<k\leq\frac{2}{3}n,\\
                                 \medskip
                                 \displaystyle{\frac{1}{\log_2[2^{2(n-k)}-1]}}&\mbox{for }\frac{2}{3}n<k\leq n-1;
\end{cases}
$$
then the $p$-Brunn Minkowski for intrinsic volumes does not hold if $p<\bar{p}_k$. In particular, this proves Theorem \ref{BM-fails}.

\begin{remark}\label{contiP}
	The value of $\bar{p}_k$ is bounded away from $1$, as $n$ and $k$ range in $\N$ and $\{2,\dots,n-1\}$, respectively. We analyse its value and its asymptotic behaviour in high dimension, in the three cases $1<k\le \frac n2, \frac n2< k\le \frac{2}{3}n\mbox{ and } \frac{2}{3}n < k\le n-1$.
	
	\medskip
	
	\noindent {\bf Case $1<k\le \frac n2$.} The sequence 
	$$
	b_k=\binom{2k}{k} 2^{-k},
	$$
	is strictly increasing, hence $b_k>b_1=1$ for $k\ge 2$, which implies that 
	$$
	\bar{p}_k= \displaystyle{\frac{k}{\log_2\binom{2k}{k}}} <1;
	$$
	moreover $\lim_{k\to\infty} \bar{p}_k = \frac 12$.
	
	\medskip
	
	\noindent {\bf Case $\frac{n}2<k\le \frac{2}{3}n$}. We notice that, since $k>1$,
	\begin{equation}\label{pk3}
	\sum_{i=1}^k\binom{2(n-k)}{i} > \binom{2(n-k)}{1}=2n-2k\ge 2n-\frac 43 n=\frac 23 n,
	\end{equation}
	hence 
	$$
	\bar{p}_k<\frac1{\log_2((2n)/3)}
	$$ 
	for every $k\le \frac{2}{3}n$ and for every $n\ge 3$.
	Hence the asymptotic behavior of $\bar{p}_k$ as $n$ tends to infinity is infinitesimal for every $\frac{n}2<k\le \frac{2}{3}n$.
	
	\medskip
	
	\noindent{\bf Case $\frac{2}{3}n<k\le n-1$.} Since $n-k\ge 1$, we have $\bar{p}_k\le 1/\log_2 3<1$.		
\end{remark}


%




\begin{thebibliography}{99}

\bibitem{Abate-Tovena} M. Abate, F. Tovena, Curves and surfaces, translated from Italian by D. A. Gewurz, Unitext, 55, Springer, Milan, 2012.

\bibitem{Bonnesen-Fenchel} T. Bonnesen, W. Fenchel. Theorie der konvexen K\"orper, Springer, Berlin, 1934. Reprint: Chelsea Publ. Co., New York, 1948. English translation: BCS Associates, Moscow, Idaho, 1987.  

\bibitem{Boroczky-De} K. B\"or\"oczky, A. De. {\em Stability of the logarithmic Brunn-Minkowski inequality in the case of many hyperplane symmetries.} Preprint; arXiv 2101.02549.

\bibitem{Boroczky-Kalantzopoulos} K. B\"or\"oczy, P. Kalantzopoulos.  {\em Log-Brunn-Minkowski inequality under symmetry.} Preprint; arXiv 2002.12239.

\bibitem{BLYZ} K.J. B\"or\"oczky, E. Lutwak, D. Yang, G. Zhang.  \emph{The log-Brunn-Minkowski inequality.} Adv. Math.  \textbf{231} (2012), no.3--4, 1974--1997.

\bibitem{BLYZ-1} K. J. B\"or\"oczky, E. Lutwak, D. Yang, G. Zhang.  \emph{The logarithmic Minkowski Problem.} J. Amer. Math. Soc.  \textbf{26} (2013), no. 3, 831--852.

\bibitem{Chen-Huang-Li-Liu} S. Chen, Y. Huang, Q.-R. Li, J. Liu.  {\em The $L_p$-Brunn-Minkowski inequality for $p<1$.} Adv. Math.  \textbf{368} (2020),  107166, 21 pp.


\bibitem{CHS} A. Colesanti, D. Hug, E. Saor\'{i}n-G\'omez. \emph{Monotonicity and concavity of integral functionals.} Commun. Contemp. Math. \textbf{19} (2017), no. 2, 1650033, 26 pp. 


\bibitem{CL} A. Colesanti, G. Livshyts. \emph{A note on the the quantitative local version of the log-Brunn-Minkowski inequality.} Advances in Analysis and Geometry 2,  special volume dedicated to the mathematical legacy of Victor Lomonosov,  (2020).

\bibitem{CLM} A. Colesanti, G. Livshyts, A. Marsiglietti. {\em On the stability of log-Brunn-Minkowski type inequality.} J.  Funct.  Anal.  \textbf{273} (2017),  no.3,  1120--1139.

\bibitem{CS} A. Colesanti, E. Saor\'in-G\'omez. \emph{Functional inequalities derived from the Brunn-Minkowski inequality for quermassintegrals.} J. Convex Anal.  \textbf{17} (2010),  n.1,  35--49.

\bibitem{CE} D. Cordero-Erausquin. {\em Santal\'o's inequality on $\mathbb{C}^n$ by complex interpolation.} C. R. Math. Acad. Sci. Paris \textbf{334} (2002), n. 9, 767--772.

\bibitem{Firey} W. J. Firey.  {\em $p$-means of convex bodies.} Math. Scand. \textbf{10} (1962), 17--24. 

\bibitem{Gar} R. Gardner. \emph{The Brunn-Minkowski inequality.} Bull. Amer. Math. Soc.  \textbf{39} (2002),  no.3, 355--405.

\bibitem{Hosle-Kolesnikov-Livshyts} J. Hosle, A. Kolesnikov, G. Livshyts.  {\em On the $L_p$-Brunn-Minkowski and the dimensional Brunn-Minkowski conjectures for log-concave measures.} Preprint; arXiv 2003.05282.  

\bibitem{Jerison1996} D. Jerison.  {\em The Minkowski problem for electrostatic capacity.} Acta Math.  \textbf{176} (1996),  1--47.

\bibitem{Koldobsky} A. Koldobsky.  Fourier analysis in convex geometry,  AMS, Providence,  2005.

\bibitem{Kolesnikov-Livshyts} A. Kolesnikov, G. Livshyts. {\em On the local version of the log-Brunn-Minkowski conjecture and some new related geometric inequalities.} Preprint; arXiv 2004.06103. 

\bibitem{Kolesnikov-Milman} A. Kolesnikov, E. Milman. \emph{Local $L^p$-Brunn-Minkowski inequalities for $p<1$.} To appear in Memoirs of the AMS; arXiv 1711.01089.

\bibitem{Lutwak1} E. Lutwak. {\em The Brunn-Minkowski-Firey theory, I: Mixed volumes and the Minkowski problem.} J. Differential Geom.  \textbf{38} (1993), no.1, 131--150.

\bibitem{Lutwak2} E. Lutwak. {\em The Brunn-Minkowski-Firey theory, II: Affine and geominimal surface areas.} Adv. Math.  \textbf{118} (1996),  no.2,  244--294.

\bibitem{Ma} L. Ma. {\em A new proof of the log-Brunn-Minkowski inequality.} Geom. Dedicata \textbf{177} (2015),  75--82.

\bibitem{Milman} E. Milman. {\em A sharp centro-affine isospectral inequality of Szeg\"o-Weinberger type and the $L_p$ Minkowski problem.} Preprint; arXiv 2103.02994.

\bibitem{Putterman} E. Putterman. {\em Equivalence of the Local and Global Versions of the $L_p$-Brunn-Minkowski Inequality.} To appear in J.  Funct. Anal.; arXiv 1909.03729.

\bibitem{Rotem} L. Rotem. {\em A letter: The log-Brunn-Minkowski inequality for complex bodies.} Unpublished; arXiv 1412.5321.

\bibitem{Saroglou} C. Saroglou. {\em Remarks on the conjectured log-Brunn-Minkowski inequality.} Geom. Dedicata \textbf{177} (2015),  353--365.

\bibitem{Schneider} R. Schneider.  Convex bodies: The Brunn-Minkowski theory, second expanded edition,  Encyclopedia of Mathematics and its Applications, Cambridge University Press, Cambridge, 2013.

\bibitem{Xi-Leng} D. Xi, G. Leng. {\em Dar's conjecture and the log-Brunn-Minkowski inequality.} J. Differential Geom.  \textbf{103} (2016), n.1, 145--189.





\end{thebibliography}
\end{document}